\documentclass[12pt]{article}
\usepackage{amsthm}
\usepackage{amssymb, url}
\usepackage{amsmath}
\usepackage{amsfonts}
\usepackage{graphicx}
\usepackage{xcolor}
\definecolor{darkgreen}{rgb}{0.0, 0.6, 0.0}
\usepackage{float}
\usepackage{multicol}
\usepackage{subcaption}
\usepackage{multirow}
\usepackage{geometry} 
\usepackage{comment}

\newtheorem{theorem}{Theorem}[section]
\newtheorem{lemma}[theorem]{Lemma}

\newtheorem{algorithm}[theorem]{Algorithm}

\newtheorem {remark}[theorem]{Remark}

\title{Data assimilation for slightly compressible flow}
\author{Aytekin \c{C}{\i}b{\i}k\thanks{Department of Mathematics, Gazi University, Ankara, T\"urkiye}, \and Rui Fang\thanks{Department of Mathematics, The Ohio State University, Columbus, OH, 43210, USA (fang.1211@osu.edu). Corresponding author: Rui Fang.}}

\begin{document}
\maketitle

\begin{abstract}
Continuous data assimilation (CDA) nudges observational data into governing equations to recover the underlying flow and improve predictions. Existing rigorous CDA analyses focus primarily on incompressible flows, yet no physical flow is perfectly incompressible. Approximating a slightly compressible flow with an incompressible model introduces non-negligible model errors. Data assimilation for compressible flows remains challenging due to strong nonlinearities and the presence of shocks. We design an algorithm that addresses the limitations of velocity-only nudging for slightly compressible flow. This work incorporates both velocity and pressure data from the slightly compressible flow and nudges both quantities into the incompressible Navier--Stokes equations. Our analysis shows that the model error decays exponentially in the initial error, with an asymptotic residual of order \(O(H)\), where $H$ denotes the observation resolution. The analysis also identifies a scaling for the pressure nudging parameter \(\mu_{1} = \mathcal{O}(1 / H^{2})\) that ensures effective assimilation. We validate the theoretical results through a suite of numerical experiments: a convergence study confirming optimal rates, a modified Taylor--Green vortex benchmark demonstrating synchronization of energy, enstrophy, and pressure, and an acoustic wave propagation test that isolates the role of pressure nudging and achieves a \(97.9\%\) reduction in pressure error relative to velocity-only assimilation. Together, these results provide a foundation for discrete error estimates and realistic compressible applications.
\end{abstract}


\section{Introduction} Predicting the evolution of fluid flows remains a fundamental challenge due to their inherent chaotic and multi-scale nature. We consider the flow of an incompressible viscous fluid in a domain $\Omega \subset \mathbb{R}^d$, with $d=2$ or $3$, governed by the Navier--Stokes equations (NSE)
\begin{align}
u_{t}+u\cdot \nabla u-\nu \triangle u+\nabla p &= f(x),\quad
\nabla \cdot u=0,\quad \text{in }\Omega,\ 0<t\leq T, \\
u &= 0\ \text{ on }\partial \Omega,\quad u(x,0)=u_{0}(x).
\end{align}
Solutions of this nonlinear system exhibit sensitivity to initial conditions, allowing small perturbations to grow rapidly in time, leading to a finite predictability horizon \cite{Lorenz1963}.

This intrinsic limitation motivates the incorporation of observational data into the model in order to improve prediction accuracy. Data assimilation incorporates observational data into the dynamical system to produce more accurate forecasts than the model or the data alone \cite{K03}. In practice, only coarse-grained measurements of the velocity field are available, which we represent as
\begin{equation*}
u_{\text{obs}}(x,t)=I_{H}u(x,t),
\end{equation*}
where $I_H$ denotes an observation (interpolation) operator associated with spatial resolution $H$.

The modern rigorous framework for continuous data assimilation (CDA) was established by Azouani, Olson, and Titi in 2014 \cite{AOT14}, who introduced a nudged approximation system for the $2d$ incompressible NSE. Their central result is that, for suitably large nudging parameter $\chi$ and sufficiently small observation resolution $H$, the nudged solution converges exponentially to the unknown true solution, yielding an infinite predictability horizon.

This foundational work spurred many directions of development. We highlight a few key directions. First, the CDA framework was extended to more complex fluid models, including the $3d$ NSE and non‑standard formulations \cite{BP21, CGJP22}. Second, numerical analysis became a major focus, particularly the design and convergence of implicit--explicit (IMEX) finite element methods \cite{LRZ19} and the treatment of large parameter regimes \cite{DLR25}. Third, practical challenges were addressed such as simultaneous parameter recovery \cite{CHLMN21}, handling of model errors \cite{CFLS25-model-error}, and adaptive selection of nudging parameters \cite{CFLS25}. For solutions with higher regularity, improved estimates allowed less restrictive nudging parameters \cite{Garcia-Archilla_Novo_2020}. Other innovations include  modular nudging to reduce computational complexity \cite{CFL26}, local nudging strategies for global recovery \cite{BBJ21, FP26}, and nonlinear nudging terms \cite{CLT24}. Together, these works (among hundreds of others) demonstrate the versatility and growing maturity of CDA as a framework for long‑time prediction in fluid dynamics.

A unifying feature of all existing rigorous CDA analyses is the restriction to
\emph{incompressible} flow, where the divergence constraint $\nabla\cdot u=0$
eliminates pressure as an independent dynamical variable;
pressure serves only as a Lagrange multiplier, so once the velocity is corrected
by nudging, the pressure adjusts instantaneously.
There is neither a need nor a  mechanism to nudge pressure.

This picture breaks down for \emph{slightly compressible} flows at low Mach
numbers.
Under the Stokes hypothesis, the slightly compressible NSE are
\begin{align}
u_t +u \cdot \nabla u +\tfrac{1}{2} (\nabla \cdot u) u
    -\nu \Delta u -\tfrac{1}{3}\nu \nabla(\nabla \cdot u) +\nabla p
    &= f(x), \label{compressible_momentum}\\
\frac{1}{c^2}\!\left(\frac{\partial p}{\partial t} + u\cdot \nabla p\right)
    +\nabla \cdot u &= 0, \label{compressible_continuity}
\end{align}
where $c$ is the isentropic speed of sound ($c^2 = K/\rho$ for liquids,
$c^2=\gamma RT$ for an ideal gas, where 
$K$ is the bulk modulus, $\rho$ the density, $\gamma$ the heat capacity ratio, $R$ the gas constant, and $T$ the temperature).

The continuity equation couples pressure and velocity through a time derivative $\frac{\partial p}{\partial t}$, nonlinear transport term $u\cdot \nabla p$, and $\nabla \cdot u$, giving
pressure its own dynamics. The sound speed $c$ (via the factor $1/c^{2}$) controls the compressibility of the system and the time scale of pressure dynamics. Acoustic waves propagate at finite speed $c$; in the limit, $c\to\infty$, the
incompressible formulation is recovered.

In practice, no physical flow is perfectly incompressible \cite{oran2002fluid}. Approximating a slightly compressible flow with an incompressible model introduces model errors. Traditional intuition suggests that small model errors should lead to proportionally small effects on a solution. But Birkhoff in 1950 \cite{Birkhoff_50} pointed out that this is a fallacious argument, stating that ``the possibility that arbitrarily small causes can produce finite effects in fluid mechanics has been clearly recognized..." Indeed, our compressible flow examples show that the model error between slightly compressible and incompressible dynamics cannot be dismissed, precisely because small compressibility can lead to a finite and significant error due to the system's strong nonlinearity.

Data assimilation for compressible flows is known to be challenging due to strong nonlinearities and the presence of shocks, which can lead to instability and spurious oscillations in standard assimilation methods \cite{zhou2026neural}. Thus, the goal in CDA is to nudge the incompressible model with compressible observations, thereby reducing model error and obtaining a consistent approximation of the target dynamics. In this setting, velocity nudging alone cannot adequately control the pressure field: forcing the nudged velocity $v$ to match the observation $u$ through the momentum equation, while leaving the pressure equation uncorrected, produces an inconsistent velocity--pressure pair that generates spurious acoustic waves via the divergence term
$\nabla\cdot v$ in \eqref{compressible_continuity}.
\subsection{Algorithm}
Motivated by the limitations of velocity-only nudging, we propose a CDA algorithm that incorporates \emph{both velocity and pressure}. The algorithm assimilates data from a slightly compressible reference flow, stated in (\ref{compressible_momentum})-(\ref{compressible_continuity}), into an incompressible Navier--Stokes model by simultaneously correcting the velocity and pressure fields toward the observations.
The assimilated system reads
\begin{align}
v_t +v \cdot \nabla v -\nu \Delta v +\nabla q
    - \chi I_H(u-v) &= f(x), \label{momentum_eqn_v1}\\
\nabla \cdot v - \mu_1 I_H(p-q) - \mu_2(I_H(q)-q) &= 0.
    \label{continuity_eqn_v1}
\end{align}
Here, $\chi$ controls the strength of velocity nudging, enforcing alignment of the velocity field with observations. The parameter $\mu_1$ governs the assimilation of pressure data by coupling the modeled pressure with observed pressure through the interpolation operator. In contrast, $\mu_2$ acts as a regularization parameter that suppresses unresolved (fine-scale) components of the pressure field by enforcing consistency between $q$ and its interpolant $I_H(q)$. When $\mu_1=\mu_2$, the nudged continuity equation is reduced to $\nabla \cdot v - \mu_1( I_H(p)-q) = 0$.

The analysis is based on a condition on $H$ and $\chi$ \cite{CFLS25}. We select $H$ and $\chi$ so that the following conditions hold:
\begin{eqnarray}\label{condi_3d}
\chi-\frac{27}{16}\frac{C_2^4}{\nu^3}\|\nabla v\|_{L^{\infty
}(0,T;L^{2}(\Omega ))}^{4} &\geq &\alpha \chi >0, \\
\nu -2\left( C_{1}H\right)^{2}\chi  &>&0, \label{h_condi}
\end{eqnarray}
for some fixed $\alpha \in (0,1)$.

We prove an error estimate (Theorem~\ref{error_finite}) showing exponential decay
of both the velocity and pressure errors under suitable conditions on $H$, $\chi$,
and the pressure nudging parameters $\mu_1$, $\mu_2$.
The key new analytical step is bounding the pressure related terms and acoustic transport terms; the optimal choice $\mu_1=\mathcal{O}(H^{-2})$ is identified from the analysis and confirmed numerically.

We conduct a suite of numerical experiments to validate the theory: a convergence study, a modified Taylor--Green vortex experiment, and an acoustic wave propagation test. In the tests, we use the second-order backward differentiation formula (BDF2) together with IMEX time discretization and Taylor-Hood ($\mathbb{P}_2 -\mathbb{P}_1$) finite element spaces for velocity and pressure. First, we confirm second-order convergence in both
velocity and pressure in time and space with a manufactured solution.
A modified Taylor--Green vortex benchmark \cite{taylor1937mechanism}, commonly used for compressible flows, shows that
energy, enstrophy, and pressure of the nudged system synchronize with the
compressible reference solution.
An acoustic wave propagation test isolates the necessity of pressure nudging:
velocity-only nudging ($\mu_1=\mu_2=0$) reduces the $L^2$ pressure error by
only $6\%$ relative to the free run with no assimilation, while full
pressure--velocity nudging achieves $97.9\%$ reduction, consistent
with the theoretical prediction.
This quantitative result reinforces a principle present in the earliest
meteorological nudging work: different state variables carry
fundamentally different information, and each must be corrected to obtain an accurate forecast \cite{SS90}.

The remainder of the paper is organized as follows: Section~2 derives the slightly compressible NSE from the general compressible
equations under the Stokes hypothesis.
Section~3 states and proves the continuous error estimates.
Section~4 presents the BDF2/IMEX finite element scheme and numerical experiments.
 
\subsection{Notations and Preliminaries}
We denote the $L^{2}(\Omega )$ norm and inner product by $\Vert \cdot \Vert $ and $(\cdot , \cdot)$, respectively. By $\Vert \cdot \Vert_{L^{p}}$ we indicate the $L^{p}(\Omega )$ norm. 

The solution spaces $X$ for the velocity and $Q$ for the pressure are defined as:
\begin{equation*}
\begin{aligned}
    &X:=(H_0^1(\Omega))^d=\{ v\in (L^2(\Omega))^d: \nabla  v\in (L^2(\Omega))^{d \times d}\ \text{and}\  v=0\ \text{on}\ \partial\Omega\},\\
    &Q:=L^2_0(\Omega)=\{q\in L^2(\Omega): \int_\Omega q\ d x=0\}.
\end{aligned}
\end{equation*}

The finite element method for this problem involves picking finite element spaces \cite{L08}, $X^h\subset X$ and $Q^h\subset Q$. We assume that $(X^h, Q^h)$ are conforming. We assume $X^h$ has approximation properties, for $u\in (H^{m+1}(\Omega))^d\ \cap\ (H_0^1(\Omega))^d$ and $p\in H^{m}(\Omega)$,
\begin{equation}\label{prop}
\begin{aligned}
\inf_{v^h\in X^h}\{\|u-v^h\|+h\|\nabla(u-v^h)\|\}&\leq Ch^{m+1}|u|_{m+1},
\\\inf_{q^h\in Q^h}\|p-q^h\|&\leq Ch^{m}|p|_{m},
\end{aligned}
\end{equation}
\begin{equation}\label{infsup}
\inf_{q^h\in Q^h}\sup_{ v^h\in X^h}\frac{(q^h,\nabla\cdot  v^h)}{\|q^h\|\|\nabla  v^h\|}\geq \beta^h>0.
\end{equation}
We also use the following inequalities. The first estimate of this type was derived in 1958 by Nash \cite{N58}. A summary of the improvements to some constants is mentioned in \cite{CFL26}.

\begin{lemma}\label{norms_property} For any vector function $u:{{{{\mathbb{R}}}}}%
^{d}\rightarrow {{{{\mathbb{R}}}}}^{d}$ with compact support and with the
indicated $L^{p}$ norms finite, 
\begin{align*}
\Vert u\Vert _{L^{4}({{{{\mathbb{R}}}}}^{2})}& \leq 2^{1/4}\Vert u\Vert^{1/2}\Vert \nabla u\Vert^{1/2}, \\
\Vert u\Vert_{L^{4}({{{{\mathbb{R}}}}}^{3})}& \leq \left(\frac{4}{3\sqrt{3}}\right)^{3/4}\Vert
u\Vert ^{1/4}\Vert \nabla u\Vert ^{3/4}, \\
\Vert u\Vert _{L^{6}({{{{\mathbb{R}}}}}^{3})}& \leq \frac{2}{\sqrt{3}}\Vert
\nabla u\Vert.
\end{align*}%
\end{lemma}
For the nonlinear term, we use the following bounds \cite{L08}, 
\begin{equation}
\begin{split}
(u\cdot \nabla v,w)& \leq C_{2}\Vert \nabla u\Vert \Vert \nabla v\Vert \Vert
\nabla w\Vert , \\
& \leq C_{2}\sqrt{\Vert u\Vert }\sqrt{\Vert \nabla u\Vert }\Vert \nabla
v\Vert \Vert \nabla w\Vert,
\end{split}%
\end{equation}%
for $u,v,w\in (H_{0}^{1}(\Omega))^d$, $d=2$ or $3$.
The standard explicitly skew-symmetrized trilinear form is denoted as $$b^*(u,v,w)=\frac{1}{2}(u\cdot \nabla v,w)-\frac{1}{2}(u\cdot \nabla w, v).$$ Further, we have
\begin{equation}
    b^{\ast }(u,v,w) =(u\cdot \nabla v,w)+\frac{1}{2}(\left( \nabla \cdot
u\right) v,w).
\end{equation}
\begin{lemma} \label{nonlinear_bound_John}
(Lemma 6.14, p. 311 - p. 312, \cite{john2016finite}) There is a $C_2<\infty$ with
\begin{equation}
\begin{split}
 b^*(u,v, w)
&\leq C_2\sqrt{\|u\|}\sqrt{\|\nabla u\|} \|\nabla v\|\|\nabla w\|,
 \end{split}
\end{equation}
for $u,v,w \in (H^1_0(\Omega))^d$, $d=2$ or $3$.
\end{lemma}

\textbf{A1 Assumption on $I_H$}: $I_{H}$ is an $L^{2}$ projection satisfying, for
all $\phi \in \left( H_{0}^{1}(\Omega )\right) ^{d},$
\begin{equation*}
\begin{split}
\|\phi -I_{H}\phi \|\leq C_{1}H\|\nabla \phi \|.
\end{split}
\end{equation*}

\begin{lemma}
(H\"older and Young's inequality\label{holderandyoung})
For any $\sigma>0$, $1\leq p\leq \infty$, $\frac{1}{p}+\frac{1}{q}=1$, the H\"older and Young's inequalities are as follows:
\begin{equation*}
    (u,v)\leq \|u\|_{L^p}\|v\|_{L^q},\ \ \text{and }
    (u,v)\leq \frac{\sigma}{p} \|u\|_{L^p}^p+
    \frac{\sigma^{-q/p}}{q} \|v\|_{L^q}^q.
\end{equation*}
\end{lemma}

\section{Model derivation for slightly compressible flows}
In this section, we derive the NSE for slightly compressible flow. We begin with the momentum equation of the NSE \cite{LL13}:
\begin{equation}
    \begin{gathered}
    \frac{\partial u}{\partial t} + (u\cdot \nabla) u 
    = - \nabla p 
    +\nabla \cdot \big[ \nu (\nabla u + (\nabla u)^T -\frac{2}{3} (\nabla \cdot u) I) \big]\\
    +\nabla \big[ \nu_{bulk} \nabla \cdot u \big] +f,
    \end{gathered}
\end{equation}
where $\nu = \mu / \rho$ is the kinematic (shear) viscosity, 
$\nu_\text{bulk} = \mu_\text{bulk} / \rho$, 
$\mu_\text{bulk}$ is the bulk viscosity, 
$\mu$ is the dynamic (shear) viscosity, 
and $\rho$ is the fluid density. We assume the Stokes hypothesis, under which the bulk viscosity is zero, giving
\begin{equation}
    \mu_\text{bulk} = \lambda + \frac{2}{3}\mu = 0,
\end{equation}
where $\lambda$ is the second viscosity coefficient and $\mu$ is the dynamic (shear) viscosity. Thus the momentum equation reduces to
\begin{equation}
\begin{gathered}
    \frac{\partial u}{\partial t} + (u\cdot \nabla) u = - \nabla p +\nabla \cdot \big[ \nu (\nabla u + (\nabla u)^T -\frac{2}{3} (\nabla \cdot u) I) \big] +f.
    \end{gathered}
\end{equation}
Since $\nu$ is a constant,
\begin{equation}
\begin{gathered}
    \nabla \cdot (\nu \nabla u)  = \nu \Delta u,\
    \nabla \cdot \left(\nu (\nabla u)^T\right) = \nu \nabla (\nabla \cdot u),\ \text{and} \
    \nabla \cdot ( \nu (\nabla \cdot u) I ) = \nu \nabla (\nabla \cdot u). 
    \end{gathered}
\end{equation}
Thus, the momentum equation can be written as (\ref{compressible_momentum}).

Next, we derive the continuity equation for slightly compressible flow. The general continuity equation is
\begin{equation}
    \frac{\partial \rho}{\partial t} + u\cdot \nabla \rho + \rho \nabla \cdot u =0. \label{con-1}
\end{equation}
For slightly compressible flow, density is a function of pressure $\rho(p)$. Hence
\begin{equation}
    \frac{\partial \rho}{\partial t} = \frac{\partial \rho}{\partial p} \frac{\partial p}{\partial t},\ \text{and}\
    \nabla \rho = \frac{\partial \rho}{\partial p} \nabla p.
\end{equation}
Using the definition of the speed of sound $\frac{\partial \rho}{\partial p} = \frac{1}{c^2}$, (\ref{con-1}) becomes
\begin{equation}
    \frac{1}{c^2} \frac{\partial p}{\partial t} + \frac{1}{c^2} u\cdot \nabla p + \rho \nabla \cdot u =0. 
\end{equation}
Since density variations are small, we approximate $\rho \nabla \cdot u$ by a reference density $\rho_0 \nabla \cdot u$:
\begin{equation}
    \rho (p) \approx \rho_0 + \frac{\partial \rho }{\partial p} (p-p_0) = \rho_0 + \frac{1}{c^2} (p-p_0) \approx \rho_0,
\end{equation}
where the last approximation holds because $\frac{1}{c^2}(p-p_0) \ll \rho_0$. Without loss of generality, we set $\rho_0 = 1$, yielding the continuity equation (\ref{compressible_continuity}).

\section{Continuous error estimates}\label{continuous-error-estimates}
We denote the velocity error $e=u-v$, and the pressure error $e_p = p-q$. In Theorem \ref{error_finite}, we present an error estimate of our CDA method in (\ref{momentum_eqn_v1})-(\ref{continuity_eqn_v1}). While existing approaches focus solely on velocity nudging, our method additionally incorporates pressure nudging. The following bounds for pressure-related terms represent a key new result. 
\begin{theorem}\label{error_finite}
Consider the method stated in (\ref{momentum_eqn_v1}),(\ref{continuity_eqn_v1}). Assume $I_H$ satisfy A1 Assumption, conditions in (\ref{condi_3d}), and
\begin{equation*}
\begin{split}
&u_t \in L^\infty(0,T;(H^1(\Omega))^d),\quad u_{tt} \in L^\infty(0,T;(L^2(\Omega))^d),\\
&u \in L^\infty(0,T;(H^1(\Omega))^d),\quad p \in L^4(0,T;L^3(\Omega))\cap L^\infty(0,T;H^1(\Omega)),\\
&q_t \in L^2(0,T;L^2(\Omega)),\quad q \in L^2(0,T;L^2(\Omega)).
\end{split}
\end{equation*}
We have the following error estimate:
\begin{equation}
\begin{gathered}
\|e(t)\|^2 +\frac{1}{c^2} \|e_p\|^2 + \frac{\nu}{2} \int_{0}^t \exp\{-\beta\left(t-t' \right)\} \|\nabla e\|^2\, dt' \\
\leq \exp\{-\beta t \}\left(\|e(0)\|^2 +\frac{1}{c^2}\|e_p(0)\|^2\right)+ C C_1^2H^2\max_{0\leq t\leq T} \|\nabla p\|^2 \\
 + \frac{1}{c^4}\frac{1}{\mu_1}\frac{8}{3} \int_{0}^T \|\nabla u\|^4\, dt +\frac{1}{c^4}\frac{1}{\mu_1}\frac{8}{3} \int_{0}^T \|\nabla p\|^4_{L^3}\, dt \\
    + \frac{4}{c^4\mu_1}\int_{0}^T\|q_t\|^2\, dt+\frac{4|\mu_2-\mu_1 |}{\mu_1}C_1^2H^2\int_{0}^T\|\nabla q\|^2\, dt,
\label{error-estimates}
\end{gathered}
\end{equation}
where $\beta=\min[\alpha \chi, \frac{\mu_1}{4}]$.
\end{theorem}
\begin{proof}
Subtracting (\ref{compressible_momentum}) from (\ref{momentum_eqn_v1}), and taking the inner product with $w\in X$, it yields
\begin{equation}
\begin{gathered}
(e_t, w) + b^*(u,u,w)-b^*(v,v,w) + \nu (\nabla e, \nabla w) -(e_p,\nabla \cdot w) \\+ \chi (I_H(e),w) + \frac{1}{3}(\nabla \cdot e, \nabla \cdot w)=0. \label{velocity_error}
\end{gathered}
\end{equation}
We subtract (\ref{continuity_eqn_v1}) from (\ref{compressible_continuity}), and add and subtract $\frac{1}{c^2}(q_t,\lambda)$. Then we take the inner product of it with any $\lambda \in Q$. Thus we have
\begin{equation}
\begin{gathered}
\frac{1}{c^2}(e_{p,t},\lambda) + \frac{1}{c^2}(q_t,\lambda) + (\nabla \cdot e, \lambda)\\
+ \mu_1(I_H  (p-q),\lambda)+\mu_2(I_H  (q)-q,\lambda)+ \frac{1}{c^2}(u\cdot \nabla p, \lambda)=0.\label{pressure_error}
\end{gathered}
\end{equation}
We choose $w=e$ and $\lambda =e_p$, add (\ref{velocity_error}) and (\ref{pressure_error}). We obtain
\begin{equation}
\begin{gathered}
    (e_t, e) + b^*(u,u,e)-b^*(v,v,e)  +\nu (\nabla e, \nabla e) 
   + \chi (I_H(e), e) \\
   +\frac{1}{3}(\nabla \cdot e, \nabla \cdot e)+\frac{1}{c^2} (e_{p,t}, e_p) + \frac{1}{c^2} (q_t, e_p) \\
   + \mu_1(I_H  (e_p),e_p) +\mu_2(I_H  (q)-q), e_p)+ \frac{1}{c^2}(u\cdot \nabla p, e_p) =0. 
  \end{gathered}
\end{equation}
Therefore,
\begin{equation}
\begin{gathered}
    \frac{1} {2} \frac{d}{dt}\left(\|e\|^2+\frac{1}{c^2}\|e_p\|^2\right) +\nu \|\nabla e\|^2+ \chi (I_H(e),e) \\+\frac{1}{3}\|\nabla \cdot e\|^2+\mu_1(I_H (e_p), e_p) +\mu_2(I_H (q)-q), e_p)\\
    = -b^*(u,u,w)+b^*(v,v,w) +\frac{1}{c^2} (q_t, e_p)-\frac{1}{c^2} (u\cdot \nabla p,e_p). \end{gathered}\label{nudging_velocity}
\end{equation}
We bound $\chi (I_H(e),e)$ as follows:
\begin{equation}
\begin{split}\label{velocity_nudging_bound}
    \chi (I_H(e),e) &= \chi (e-e+I_H(e),e) \\
    &\geq \chi \|e\|^2 - \chi \|I-I_H(e)\|\|e\|\\
    &\geq \frac{\chi}{2} \|e\|^2- \frac{1}{2}\chi C_1^2 H^2\|\nabla e\|^2. 
\end{split}
\end{equation}
The nonlinear term is treated as follows:
\begin{equation}
    b^*(u,u,w)-b^*(v,v,w) = b^*(e,u,e) + b^*(v,e,e)=b^*(e,u,e).
\end{equation}
\begin{equation}
\begin{split}
b^*(e,u,e) &\leq C_2 \|e\|^{1/2}\|\nabla e\|^{3/2}\|\nabla u\|\\
&\leq \frac{\nu}{2} \|\nabla e\|^2 +\frac{27}{32}\frac{C_2^4}{\nu^3} \|\nabla u\|^4 \|e\|^2.
\end{split}
\end{equation}

While existing approaches focus solely on velocity nudging, our method includes pressure nudging. The following bounds for pressure-related terms are key novel results. First, we rearrange $\mu_1(I_H  (e_p),e_p)+\mu_2(I_H  (q)-q,e_p)$:
\begin{equation}
\begin{gathered}
   \mu_1(I_H  (e_p),e_p)+\mu_2(I_H  (q)-q,e_p) \\
   = \mu_1 (I_H (p)-q, e_p) + (\mu_2-\mu_1) \left((I_H -I) q, e_p\right).
   \end{gathered}
\end{equation}
We handle $\mu_1 (I_H (p)-q, e_p) $ and $(\mu_2-\mu_1) \left((I_H -I) q, e_p\right)$ in the following:
\begin{equation}
\begin{split}
    \mu_1 (I_H (p)-q, e_p) = \mu_1 (e_p, e_p) - \mu_1 ((I-I_H )(p),e_p)\\
    \geq \frac{\mu_1}{2} \|e_p\|^2 -\frac{\mu_1}{2}C_1^2H^2\|\nabla p\|^2, \text{ and } 
    \end{split}
\end{equation}
\begin{equation}
    (\mu_2-\mu_1) \left((I_H-I) q, e_p \right) \leq \frac{2|\mu_2-\mu_1 |}{\mu_1}C_1^2H^2\|\nabla q\|^2 + \frac{\mu_1}{8}\|e_p\|^2.
\end{equation}
Next, we bound the term $-\frac{1}{c^2} (q_t, e_p)$:
\begin{equation}
    \begin{split}
    -\frac{1}{c^2} (q_t, e_p) \leq \frac{1}{c^2} \|q_t\| \|e_p\|\leq \frac{2}{c^4\mu_1}\|q_t\|^2 + \frac{\mu_1}{8}\|e_p\|^2.
    \end{split}  
\end{equation}
Last, we bound the term $-\frac{1}{c^2} (u\cdot \nabla p,e_p)$:
\begin{equation}
\begin{gathered}\label{nudging_pressure_2}
    -\frac{1}{c^2} (u\cdot \nabla p, e_p)\leq \frac{1}{c^2} \|u\|_{L^6} \|\nabla p\|_{L^3}\|e_p\| \\
    \leq \frac{2}{\sqrt{3}}\frac{1}{c^2}\|\nabla u\| \|\nabla p\|_{L^3} \|e_p\| \\
    \leq \frac{1}{c^4}\frac{8}{3}\frac{1}{\mu_1} \|\nabla u\|^2 \|\nabla p\|^2_{L^3} + \frac{\mu_1}{8}\|e_p\|^2\\
    \leq \frac{1}{c^4}\frac{1}{\mu_1}\frac{4}{3} \left(\|\nabla u\|^4 + \|\nabla p\|^4_{L^3}\right) + \frac{\mu_1}{8}\|e_p\|^2.
\end{gathered}
\end{equation}
Collecting the terms from (\ref{nudging_velocity}) to (\ref{nudging_pressure_2}) and multiplying them by 2, it yields
\begin{equation}
\begin{split}
\frac{d}{dt}\left(\|e\|^2+\frac{1}{c^2} \|e_p\|^2\right) + \frac{1}{2} \nu \|\nabla e\|^2+ \frac{1}{2}(\nu -2\chi C_1^2 H^2) \|\nabla e\|^2\\
+ \left(\chi - \frac{27}{16} \frac{C_2^4}{\nu^3}\|\nabla u\|^4 \right)\|e\|^2 +\frac{\mu_1}{4}\|e_p\|^2 \leq \mu_1 C_1^2H^2\|\nabla p\|^2 + \frac{4}{c^4\mu_1}\|q_t\|^2\\
\frac{4|\mu_2-\mu_1 |}{\mu_1}C_1^2H^2\|\nabla q\|^2 +\frac{1}{c^4}\frac{1}{\mu_1}\frac{8}{3} \left(\|\nabla u\|^4 + \|\nabla p\|^4_{L^3}\right).
\label{collect-terms}
\end{split}
\end{equation}
We use the condition in (\ref{condi_3d}) and (\ref{h_condi}), and multiply both sides of (\ref{collect-terms}) by the integration factor $\exp\{\beta t\}$ , where $\beta =\min\{\alpha \chi, \frac{\mu_1}{4}\}$. Integrating from $0$ to $t$, and multiplying $\exp\{-\beta t \}$ to both sides of (\ref{collect-terms}) gives
\begin{equation}
\begin{gathered}
\|e(t)\|^2 +\frac{1}{c^2} \|e_p\|^2 + \frac{\nu}{2} \int_{0}^t \exp\{-\beta\left(t-t' \right)\} \|\nabla e\|^2\, dt' \\
\leq \exp\{-\beta t \}\left(\|e(0)\|^2 +\frac{1}{c^2}\|e_p(0)\|^2\right) + \int_{0}^t \exp\{ -\beta (t-t')\} \\
\bigg[
 \mu_1 C_1^2H^2\|\nabla p\|^2 + \frac{4}{c^4\mu_1}\|q_t\|^2
+\frac{4|\mu_2-\mu_1 |}{\mu_1}C_1^2H^2\|\nabla q\|^2 \\
+\frac{1}{c^4}\frac{1}{\mu_1}\frac{8}{3} \left(\|\nabla u\|^4 + \|\nabla p\|^4_{L^3}\right) \bigg].
\end{gathered}
\end{equation}
We handle $\int_{0}^t \exp\{ -\beta (t-t')\} \mu_1 C_1^2 H^2 \|\nabla p\|^2\, dt'$ term in the following:
\begin{equation}
\begin{gathered}
    \int_{0}^t \exp\{ -\beta (t-t')\} \mu_1 C_1^2 H^2 \|\nabla p\|^2\, dt' 
\\
\leq \max_{0\leq t \leq T} C_1^2 H^2\|\nabla p\|^2 \mu_1 \exp\{-\beta t\} \int_{0}^t \exp{\beta t'}\, dt'\\
=\max_{0\leq t \leq T} C_1^2 H^2\|\nabla p\|^2 \mu_1 \exp\{-\beta t\} \frac{\exp{\beta t}-1}{\beta}\\
=\max_{0\leq t \leq T} C_1^2 H^2\|\nabla p\|^2 \mu_1\frac{1-\exp\{-\beta t\}}{\beta}\\
\leq C \max_{0\leq t \leq T} C_1^2 H^2\|\nabla p\|^2 \left( 1-\exp\{-\beta t\}\right)\\
\leq C C_1^2 H^2 \max_{0\leq t \leq T}\|\nabla p\|^2.
\end{gathered}
\end{equation}
Therefore, we have the final result in (\ref{error-estimates}).
\end{proof}

\section{Numerical tests} 
In this section, we present numerical experiments to validate our CDA method. 
All simulations are performed using a BDF2, combined with an IMEX treatment of the nonlinear term. We use public license finite element software FreeFEM \cite{Hecht12} for the tests implemented here.  
We use an inf‑sup stable Taylor–Hood finite elements ($\mathbb{P}_2$ for velocity, $\mathbb{P}_1$ for pressure) in all the tests. The nudged system is defined as follows:
\begin{algorithm}
Given initial condition $v^0_h \in X^h$, find $(v^{n+1}_h, q^{n+1}_h) \in (X^h, Q^h)$, satisfying: 
\begin{equation}
\begin{split}
(\frac{3 v^{n+1}_h - 4v^{n}_h+v^{n-1}_h}{2\Delta t}, w^h) + \nu (\nabla v^{n+1}_h, \nabla w^h) + b^*( 2v^{n}_h - v^{n-1}_h, v^{n+1}_h, w^h)\\
+ \chi (I_H (v^{n+1}_h-u(t_{n+1}), w^h)
-(q^{n+1}_h, \nabla w^h)+ (r^h, \nabla \cdot v^{n+1}_h)
\\
+ \mu_1 (I_H(p(t_{n+1})-q^{n+1}_h), r^h) -\mu_2 (I_H(q^{n+1}_h)-q^{n+1}_h, r^h) =(f^{n+1}_h, w^h),
\end{split}
\end{equation}
for all $(w^h, r^h) \in (X^h, Q^h)$.

\end{algorithm}
\subsection{Test for accuracy}
We first verify the accuracy of the continuous data assimilation method (\ref{momentum_eqn_v1})--(\ref{continuity_eqn_v1}) using a manufactured analytical solution. Consider the square domain $\Omega = (0,1)\times(0,1)$ and the time interval $[0,T]$ with $T=2$. Following \cite{CFL26}, we construct velocity and pressure fields that satisfy the slightly compressible NSE (\ref{compressible_momentum})--(\ref{compressible_continuity}). The body force $f(x,t)$ is then obtained by substituting the above expressions into the momentum equation (\ref{compressible_momentum}).

Let $h = 1/n$ denote the mesh width, where $n$ is the number of elements per side. We employ a triangular mesh with barycentric refinement across a sequence of mesh widths, where $n = 8, 16, 32, 64$. We consider the observation operator $I_H$ to be an $L^2$-projection onto $\mathbb{P}_0$ functions, satisfying A1 Assumption with $H = h_{\mathrm{true}}$.

\subsubsection{Pressure accuracy}
In this test, we construct a pressure field that satisfies $(\ref{compressible_continuity})$ while keeping the velocity simple:
\begin{equation}
\begin{gathered}
u(x,y,t) = -\epsilon e^{t}(x,y),\ \\
p(x,y,t) = c^{2}\big[2\epsilon(e^{t}-1) + \sin\big(e^{\epsilon(e^{t}-1)}x\big) + P_0\big],
\end{gathered}
\end{equation}
where $\epsilon = 0.001$, the isentropic speed of sound $c = 10$, and $P_0 = 0.001$ is a reference pressure constant. This construction exactly satisfies the continuity equation (\ref{compressible_continuity}).

The timestep size is fixed at $\Delta t = 1/64$, which is sufficiently small to make temporal errors negligible compared to spatial discretization errors. The velocity nudging parameter is set to $\chi = 100$, while the pressure nudging parameters $\mu_1,\mu_2$ are chosen according to the analysis: $\mu_1 = \mu_2 = n^2$, i.e. $\mu_1 = 64, 256, 1024, 4096$ for $n = 8, 16, 32, 64$, respectively. 

Table \ref{tab:scaled_mu} reports the velocity and pressure errors at the final time $T=2$ together with the estimated convergence rates. Theorem \ref{error_finite} suggests taking $\mu_1 =\mathcal{O}(H^2)$ to achieve optimal convergence. The velocity and pressure errors are consistent with the analysis and the pressure‑nudging contribution remains well balanced with the other error components, preserving the optimal accuracy expected from the finite element spaces.
\begin{table}[htbp]
\centering
\begin{tabular}{c c c c c}
\hline
$n$ & $\|v^h - u\|$ & Rate & $\|q^h - p\|$ & Rate \\
\hline
$8$ & $5.31352\times 10^{-4}$ & --- & $4.04032\times 10^{-2}$ & --- \\
$16$ & $1.07629\times 10^{-4}$ & $2.30$ & $1.00944\times 10^{-2}$ & $2.00$ \\
$32$ & $2.51825\times 10^{-5}$ & $2.10$ & $2.52187\times 10^{-3}$ & $2.00$ \\
$64$ & $6.33607\times 10^{-6}$ & $1.99$ & $6.30158\times 10^{-4}$ & $2.00$
\end{tabular}
\caption{(Test for accuracy) Velocity and pressure errors at the final time $T=2$ with optimally scaled pressure nudging $\mu_1=\mu_2=n^2$. Here $\chi=100$, and $\Delta t=1/64$.}
\label{tab:scaled_mu}
\end{table}
\subsubsection{Velocity accuracy}
The velocity and pressure fields are defined as follows:
\begin{equation}
\begin{gathered}
    u(x,y,t) = -\epsilon e^t (x^2,y^2),\\
    p(x,y,t) = 2c^2\epsilon e^t (x+y),
    \end{gathered}
\end{equation}
where $c$ is the isentropic speed of sound. In this test, we simplify the pressure equation of the compressible flow: we omit $u\cdot \nabla p$. We set $c=1000$, $\epsilon=1$, $\nu=1$, $T=2$. Table \ref{velocity-accuracy} shows that the velocity and pressure errors achieve the expected second-order temporal and spatial convergence rates, confirming the correctness of the implementation.
\begin{table}[htbp]
\centering
\begin{tabular}{c c c c c c c}
\hline
 $n$ & $\Delta t$ &$\|u-v^h\|$ & Rate & $\|p-q^h\|$ & Rate\\
\hline
8 & $1/8^2$  & 0.41489 &  -- &  0.243003 &--\\
16& $1/16^2$  & 0.144845  & 1.52 & 0.061354& 1.99 \\
32 & $1/32^2$ & 0.042897 & 1.76 & 0.0154438& 1.99\\
64 & $1/64^2$  & 0.0116943 &  1.88 &0.0038777 & 1.99\\
\hline
\end{tabular}
\caption{(Test for accuracy) Errors for $u$ and $p$ for different mesh resolutions $n$ with the barycenter refinement and time-step size $\Delta t$ at the final time. $\Delta t = 1/n^2$, $\chi = n^2$, and the pressure nudging parameters are $\mu_1 = \mu_2 = n^2$.}
\label{velocity-accuracy}
\end{table}

\begin{remark}
When we fix the number of mesh points per side $n$ and the nudging parameters $\chi$, $\mu_1$, and $\mu_2$, reducing the timestep size $\Delta t$ does not improve the error between $v^h$ and the reference solution $u$. This indicates that the total error is dominated by modeling or nudging effects rather than temporal discretization. However, to estimate the temporal convergence rate, one may consider successive differences between numerical solutions obtained with decreasing time steps. 
Specifically, for time steps $\Delta t$, $\tau \Delta t$, and $\tau^2 \Delta t$, 
the relation \cite{han2025turbulence,olof_convergence}:
\begin{equation}
    \frac{\|v^h_{\Delta t} - v^h_{\tau \Delta t}\|}
         {\|v^h_{\tau \Delta t} - v^h_{\tau^2 \Delta t}\|}
    = \tau^{-p} + \mathcal{O}(\Delta t)
\end{equation}
can be used to estimate the temporal convergence order $p$. For example, choosing $\tau = 0.5$ corresponds to halving the time step. A rigorous analysis of this behavior for the fully discrete scheme remains an open problem.
\end{remark}

\subsection{Modified Taylor-Green vortex experiment}
\label{subsec:GreenTaylor} The Taylor-Green vortex is an unsteady flow of a decaying vortex which has an analytical solution for incompressible flow \cite{taylor1937mechanism}. The Taylor-Green vortex is frequently used as an initial condition or benchmark in low Mach number compressible simulations \cite{lusher2021tgv,dzanic2025supersonic}. In this test, we use a modified Taylor–Green vortex as the exact initial condition for the compressible flow model (\ref{compressible_momentum})-(\ref{compressible_continuity}).

\textit{Setup}. This variant is defined on a unit square domain, $\Omega = (0,1)\times (0,1)$, with no-slip boundary conditions for simplification of the setup \cite{fang2025numerical}. The external force is set to zero. The initial velocity and pressure of the slightly compressible flow are defined as follows:
\begin{equation}
\begin{gathered}
    u(x,y) = (\sin(2 \pi x) \cos(2\pi y), \ \sin(2 \pi y) \cos(2\pi x))^\top,\\
    p(x,y) = \frac{1}{4} (\cos(4\pi x) +\cos(4\pi y).
\end{gathered}
\end{equation}
initialize the nudged system described in (\ref{momentum_eqn_v1})-(\ref{continuity_eqn_v1}) to rest at time $t=0$. We set the number of mesh points per side  $n=64$ on the unit square, then we use barycenter refinement of the mesh. We set the timestep size $\Delta t = 0.01$, the final time $T=2$, and nudging parameters $\chi=n^2$, $\mu_1 = \mu_2 =n^2$.

\textit{Flow statistics}. To evaluate the simulations, we compute the following quantities:
\begin{itemize}
\item Kinetic energy: $\frac{1}{2} \int_\Omega |\mathbf{u}|^2 \,d\mathbf{x}$.
\item Vorticity: $\omega=\frac{\partial u_2}{\partial x} - \frac{\partial u_1}{\partial y}$, where $u(x,y)=(u_1(x,y),u_2(x,y))^\top$.
\item Enstrophy: $\frac{1}{2} \int_\Omega |\omega|^2 \, d\mathbf{x}$.
\item Divergence: $ \|\nabla \cdot u\|$.
\item Relative velocity error: $\frac{\|v^h-u^h\|}{\|u^h\|}$, and relative pressure error: $\frac{\|q^h-p^h\|}{\|p^h\|}$.
\end{itemize}

In Figure \ref{fig:flow-statistics}, we compute the energy, enstrophy, and divergence as functions of time. We also monitor the pressure evolution for both the nudged system (denoted as the numerical solution) and the slightly compressible flow (reference solution). With the nudging intensity chosen as $\chi = n^2$ and $\mu_1 = \mu_2 = n^2$, as suggested by the analysis, we observe that the energy and enstrophy remain well synchronized. The pressure of the nudged system agrees closely with that of the true (slightly compressible) model. The divergence of the nudged system increases rapidly due to the data assimilation term, exhibiting an initial overshoot in which the nudged system attains a larger divergence than the slightly compressible flow. However, the discrepancy between the divergence of the nudged system and that of the reference solution decreases over time. These results indicate that nudging both pressure and velocity data into the assimilated system is effective.
\begin{figure}[htbp]
    \centering
    \begin{subfigure}[b]{0.45\textwidth}
        \centering
    \includegraphics[width=0.95\textwidth]{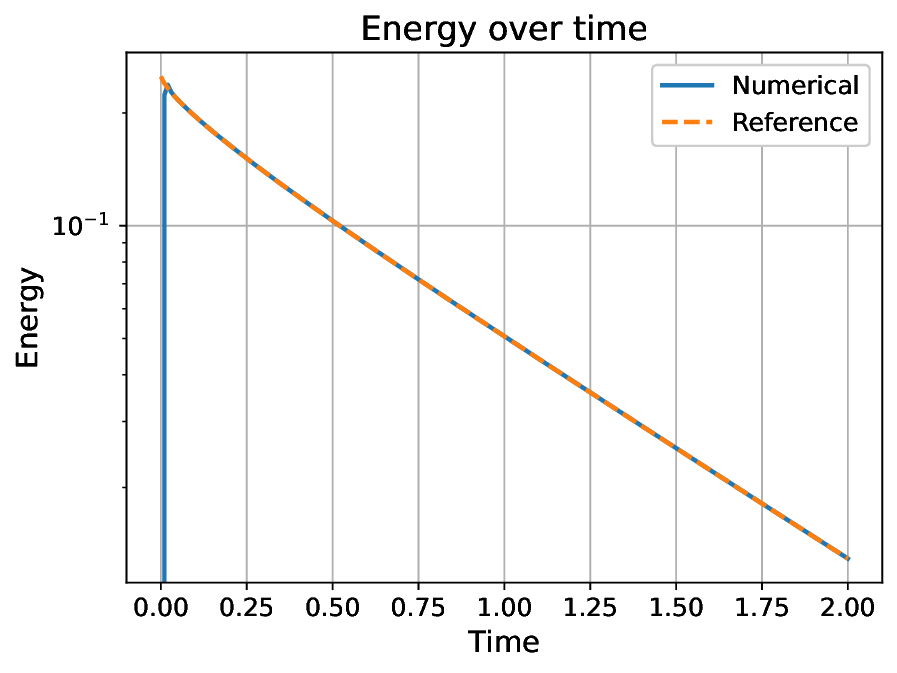}
        \caption{Energy.}
    \end{subfigure}
    \hspace{0.02\textwidth}
    \begin{subfigure}[b]{0.45\textwidth}
        \centering
        \includegraphics[width=0.95\textwidth]{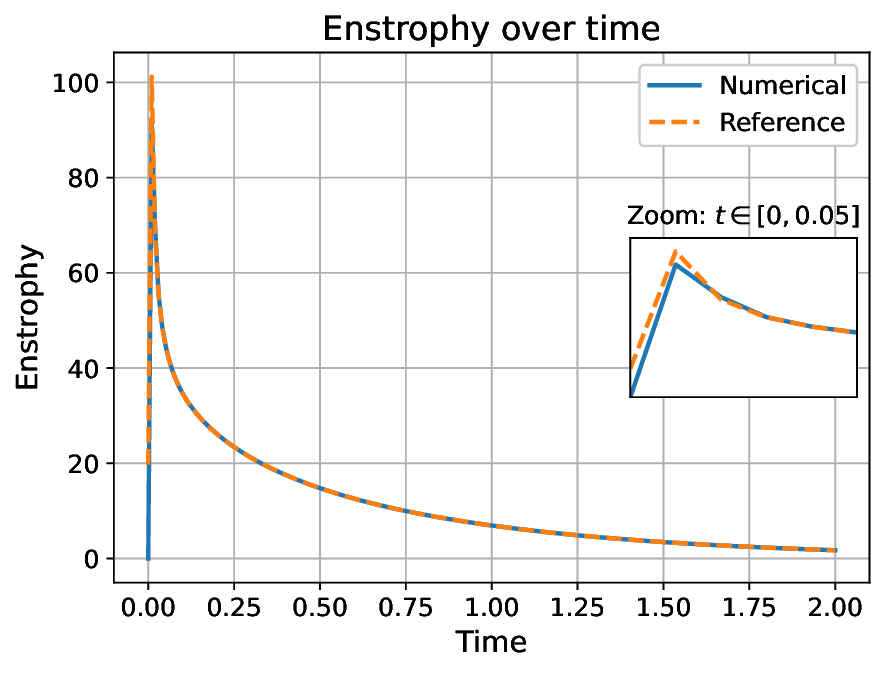}
        \caption{Enstrophy.}
    \end{subfigure}

    \begin{subfigure}[b]{0.45\textwidth}
        \centering
\includegraphics[width=0.95\textwidth]{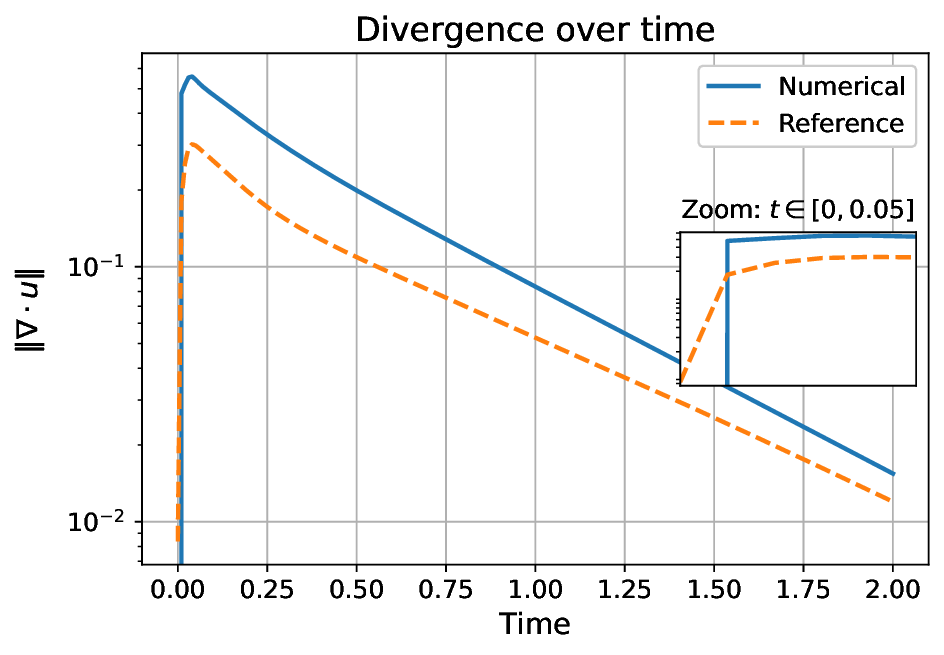}
        \caption{Divergence.}
    
    \end{subfigure}
    \hspace{0.02\textwidth}
    \begin{subfigure}[b]{0.45\textwidth}
        \centering
\includegraphics[width=0.95\textwidth]{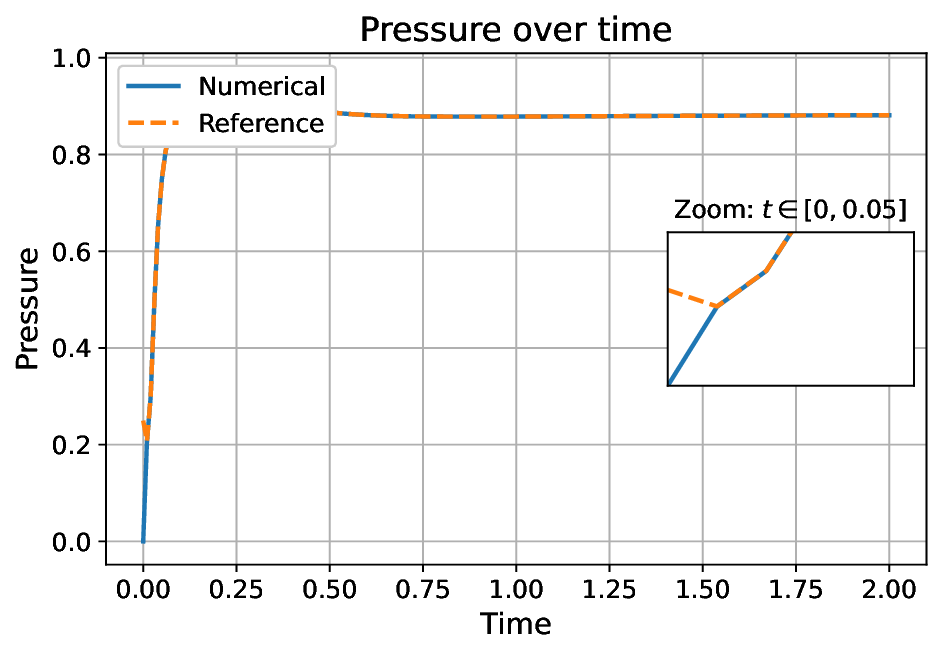}
        \caption{Pressure. }

    \end{subfigure}
    
    \caption{(Modified Taylor-Green) Time evolution of energy, enstrophy, divergence, and pressure for the nudged system and the slightly compressible reference solution. With $\chi = n^2$ and $\mu_1 = \mu_2 = n^2$, the energy and enstrophy remain synchronized, and the pressures agree closely. The divergence of the nudged system exhibits an initial overshoot but remains controlled, and its discrepancy from the reference solution decreases over time.}
    \label{fig:flow-statistics}
\end{figure}

We also present the relative velocity and pressure errors over time in Figure \ref{fig:test-2-error}. Both the relative velocity and pressure errors decrease with time, starting from $\mathcal{O}(1)$ and decreasing steadily in the intermediate regime to reach $\mathcal{O}(10^{-4})$ at the final time $t=2$. In Figure \ref{fig:velocity_vorticity}, we plot the velocity and vorticity at the final time $T=2$. The velocity fields of the nudged system and the reference solution are similar. The vorticity field of the nudged system exhibits a four-vortex structure characteristic of the modified Taylor-Green vortex, in agreement with the reference solution. However, when $\mu_1\ll n^2$, we observe that the vorticity field does not align well with the reference solution, further validating our analysis.
\begin{figure}[htbp]
    \centering    \includegraphics[width=0.5\linewidth]{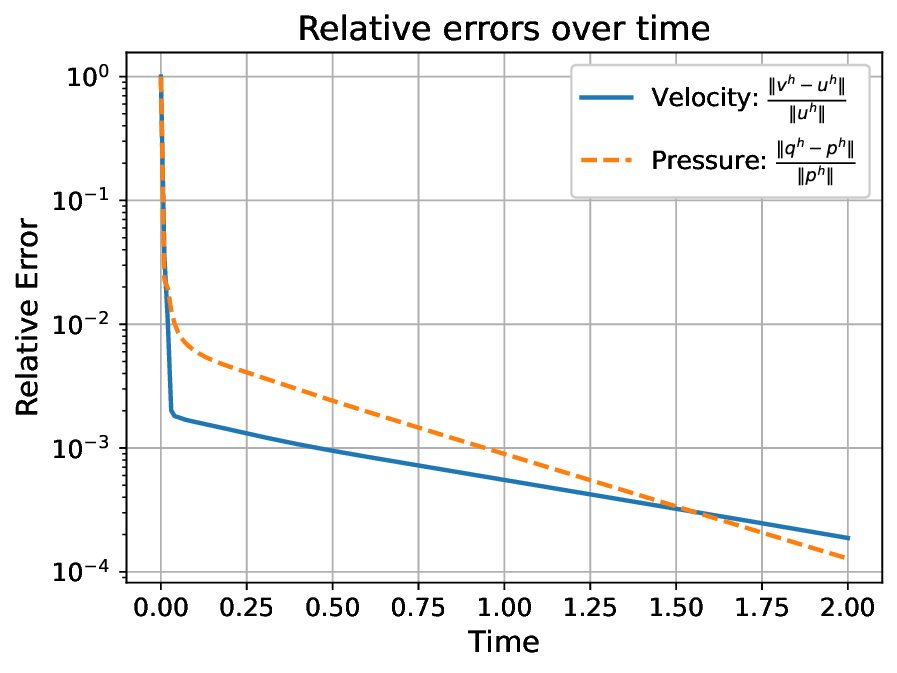}
    \caption{(Modified Taylor-Green) Relative velocity and pressure errors over time, showing decay from $\mathcal{O}(1)$ to $\mathcal{O}(10^{-4})$.}
    \label{fig:test-2-error}
\end{figure}

\begin{figure}[htbp]
    \centering
    \begin{subfigure}[b]{0.45\textwidth}
        \centering
    \includegraphics[width=0.95\textwidth]{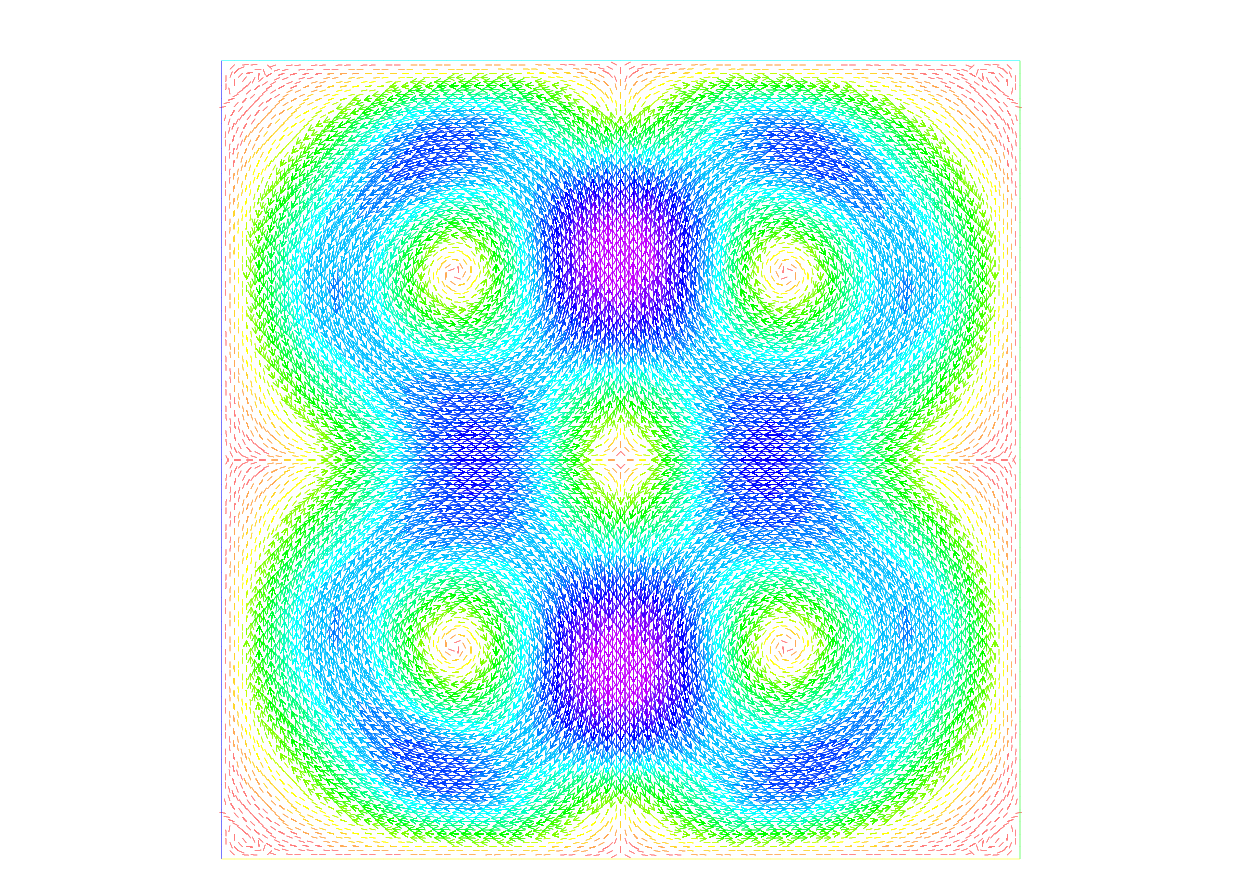}
        \caption{Numerical velocity $v^h$ at $t=2$.}
    \end{subfigure}
    \begin{subfigure}[b]{0.45\textwidth}
        \centering
        \includegraphics[width=0.95\textwidth]{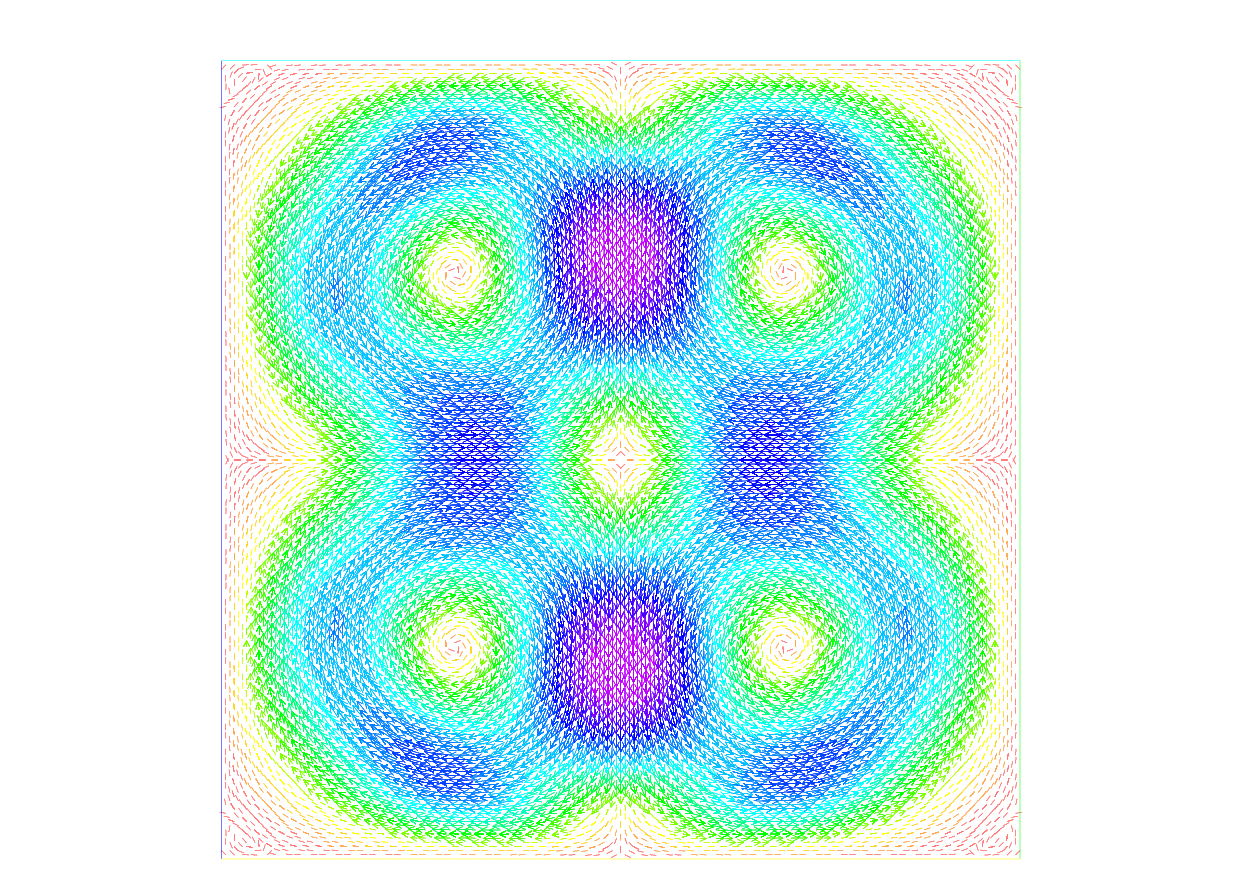}
        \caption{Reference velocity $u^h$ at $t=2$.}
        \label{fig:vel_ref}
    \end{subfigure}
    \begin{subfigure}[b]{0.45\textwidth}
        \centering
\includegraphics[width=0.95\textwidth]{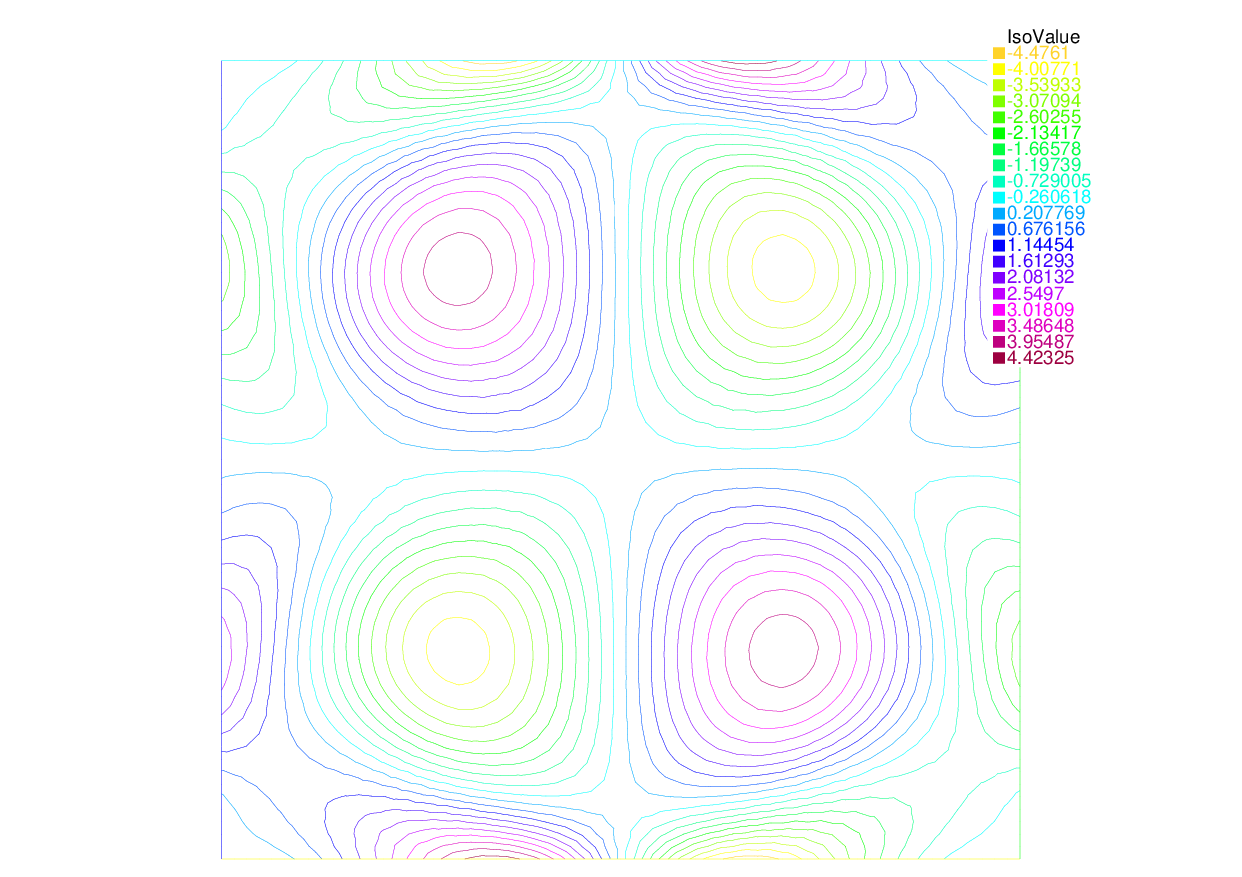}
        \caption{Numerical vorticity at $t=2$.}
        \label{fig:vort_num}
    \end{subfigure}
    \begin{subfigure}[b]{0.45\textwidth}
        \centering
        \includegraphics[width=0.95\textwidth]{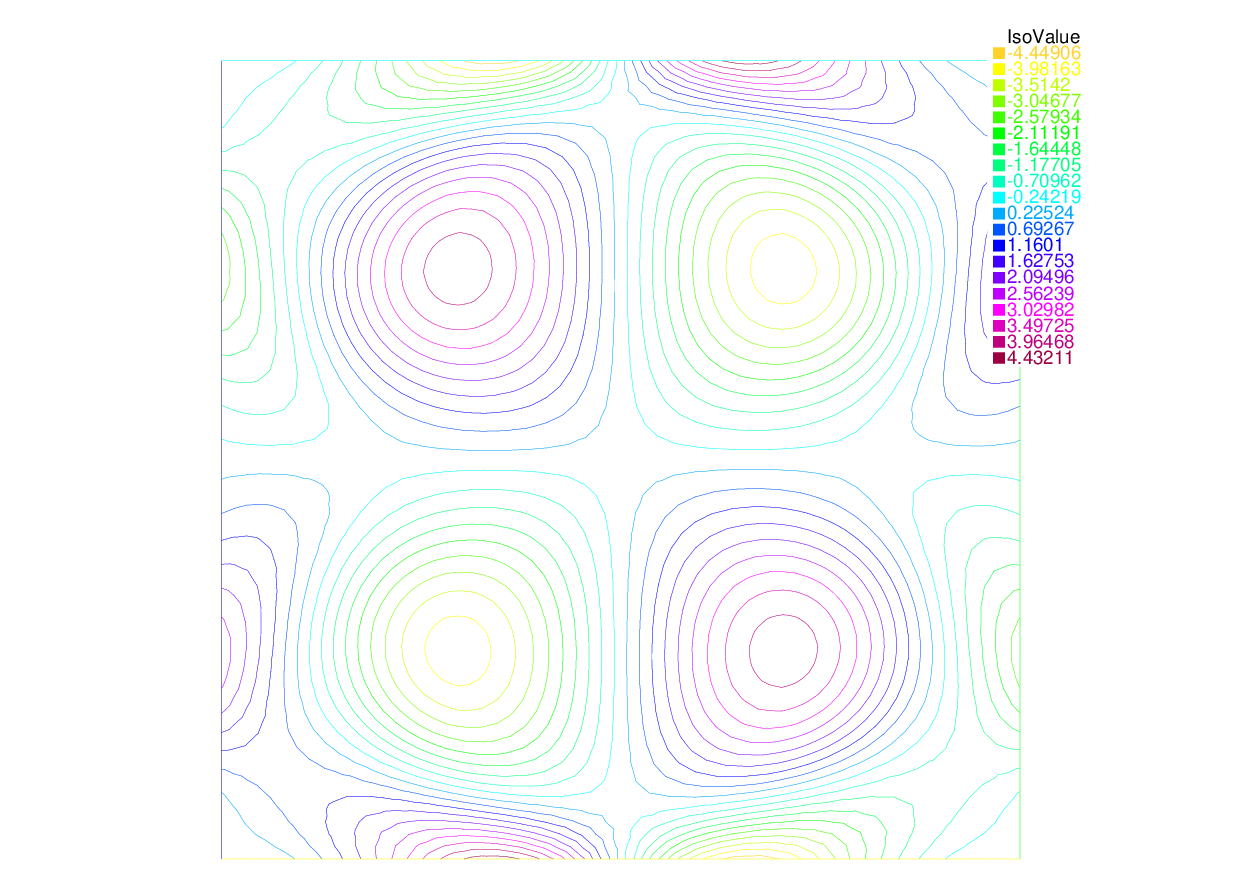}
        \caption{Reference vorticity at $t=2$. }
        \label{fig:vort_ref}
    \end{subfigure}
    \caption{(Modified Taylor-Green) Comparison of velocity and vorticity fields at $t=2$.}
    \label{fig:velocity_vorticity}
\end{figure}

\subsection{Pressure nudging in acoustic waves}
\label{subsec:acoustic}

Slightly compressible flows support acoustic wave propagation---a feature
absent from incompressible models. The initial Gaussian pressure pulse
\eqref{eq:acoustic_ic} is a standard benchmark in computational aero-acoustics,
introduced in \cite{Hardin1995} and widely used to validate compressible
solvers \cite{TamWebb1993}. This test isolates and demonstrates the \emph{necessity} of the pressure nudging parameters $\mu_1$ and $\mu_2$ for reconstructing pressure wave dynamics.

\subsubsection{Setup}
We consider $\Omega = (0,10)\times (0,10)$ with an initial Gaussian pressure pulse centered at $(x_c,y_c) = (5,5)$:
\begin{equation}\label{eq:acoustic_ic}
  p(x,y,0) = P_0 + \delta P\exp\!\left(-\frac{(x-x_c)^2+(y-y_c)^2}{2\sigma^2}\right), \qquad u(x,y,0) = 0,
\end{equation}
with $P_0 = 10^5$~Pa, $\delta P = 1$~Pa ($\delta P/P_0 = 10^{-5}$), $\sigma = 0.5$~m, $c = 1$~m/s, and $\nu = 10^{-3}$~m$^2$/s. This initial condition follows the standard form used in computational aero-acoustics benchmarks \cite{Hardin1995,TamWebb1993,Bakhvalov2025}. We integrate to $T = 3.5$~s so that the wave front remains safely inside the domain. The amplitude $\delta P = 1$~Pa keeps the acoustic Courant–Friedrichs–Lewy (CFL) number $c \Delta t/h \approx 0.64$, compatible with the implicit BDF2 scheme.

The time step is $\Delta t = 0.05$~s and the nudging parameters are
\begin{equation}\label{eq:acoustic_nudging_params}
  \chi = \frac{1}{\Delta t} = 20, \qquad \mu_1 = \mu_2 = \frac{1}{c^2\,\Delta t} = 20.
\end{equation}

The true solution is computed on a fine mesh ($n_{\mathrm{true}} = 128$, $h_{\mathrm{true}} \approx 0.078$~m); the assimilation system uses a coarser mesh ($n_{\mathrm{assim}} = 32$, $h_{\mathrm{assim}} = 0.3125$~m).

\subsubsection{Experimental cases}
All assimilation cases start from the \emph{wrong} initial condition (IC) ($p\equiv P_0$, $u\equiv 0$---no pulse), so the nudging must reconstruct the wave entirely from observations:
\begin{itemize}
  \item \textbf{TRUE}: reference slightly compressible solution \eqref{compressible_momentum}--\eqref{compressible_continuity} on the fine mesh with correct IC \eqref{eq:acoustic_ic}.
  \item \textbf{FREE}: assimilation system with $\chi = \mu_1 = \mu_2 = 0$ and wrong IC (no observations).
  \item \textbf{VEL}: velocity nudging only ($\chi = 20$, $\mu_1 = \mu_2 = 0$), wrong IC.
  \item \textbf{FULL}: velocity \emph{and} pressure nudging ($\chi = \mu_1 = \mu_2 = 20$), wrong IC.
\end{itemize}

\subsubsection{Results}
We monitor probe pressure histories $q(x_p,y_c,t)-P_0$ at $x_p \in \{7,8\}$~m along $y = y_c = 5$~m (theoretical arrival times $t_{\mathrm{arr}} = 2$~s and $3$~s, respectively) and the global $L^2$ pressure error $\|q - p_{\mathrm{true}}\|_{L^2(\Omega)}$. The wave speed is estimated from peak arrival times. Table~\ref{tab:acoustic} and Figure~\ref{fig:acoustic} summarize the results.
\begin{table}[h!]
\centering
\begin{tabular}{l c c}
\hline
Case & Final $\|q-p_{\rm true}\|_{L^2}$ & Reduction vs.\ FREE \\
\hline
FREE  & $6.077\times 10^{-1}$ & ---    \\
VEL   & $5.707\times 10^{-1}$ & $6.1\%$ \\
FULL  & $1.263\times 10^{-2}$ & $97.9\%$ \\
\hline
\end{tabular}
\caption{(Acoustic wave propagation) Final $L^2$ pressure error at $T=3.5$ s. Parameters:
$c=1$ m/s, $\delta P = 1$ Pa, $\sigma = 0.5$ m, $\chi = \mu_1 = \mu_2 = 20$,
$n_{\rm true}=128$, $n_{\rm assim}=32$, $\Delta t = 0.05$ s.}
\label{tab:acoustic}
\end{table}

\textbf{FREE} produces no wave reconstruction: the pressure remains flat at all probes for all $t\in[0,T]$.

\textbf{VEL} (velocity nudging only) captures the correct wave shape and speed at the probes, but the global $L^2$ pressure error is only $6.1\%$ better than FREE. The reason is fundamental: when $\mu_1 = \mu_2 = 0$, the assimilation continuity equation
\begin{equation}\label{eq:continuity_nudged}
   \nabla\cdot v = 0
\end{equation}
receives no direct pressure correction. Velocity nudging drives $v\to u$ through the momentum equation, but without a corresponding pressure correction, the divergence $\nabla\cdot v$ acts as a \emph{spurious acoustic source}. When $\mu_1 = \mu_2 = 0$, the assimilation continuity equation
\eqref{continuity_eqn_v1} receives no direct pressure correction,  which explains why 
the global $L^2$ pressure error of VEL is only $6.1\%$ better than FREE 
despite the probe histories looking reasonable. This is a fundamental limitation of velocity-only nudging in the slightly compressible setting.

\textbf{FULL} (velocity and pressure nudging) reduces the $L^2$ error by $97.9\%$ relative to FREE. With $\mu_1 = \mu_2 = 20$, the right-hand side of \eqref{eq:continuity_nudged} is replaced by $\mu_1 I_H(p-q) + \mu_2(I_H(q)-q)$, which directly supplies the observed pressure signal, creating an acoustic source consistent with the true wave. The probe histories are visually indistinguishable from TRUE, and the $L^2$ error decreases rapidly in the first $0.5$~s as the nudging injects the correct pressure field.

\begin{remark}
The necessity of pressure nudging is specific to the slightly compressible setting. In the incompressible limit ($c\to\infty$), the continuity equation reduces to $\nabla\cdot v = 0$ and the pressure is a Lagrange multiplier with no independent dynamics; velocity nudging alone suffices \cite{AOT14,LRZ19}. The compressibility parameter $1/c^2$ coupling pressure and velocity dynamics is precisely what makes the $\mu_1,\mu_2$ terms in \eqref{continuity_eqn_v1} necessary rather than merely beneficial.
\end{remark}

\begin{figure}[htbp]
    \centering
\includegraphics[width=0.9\textwidth]{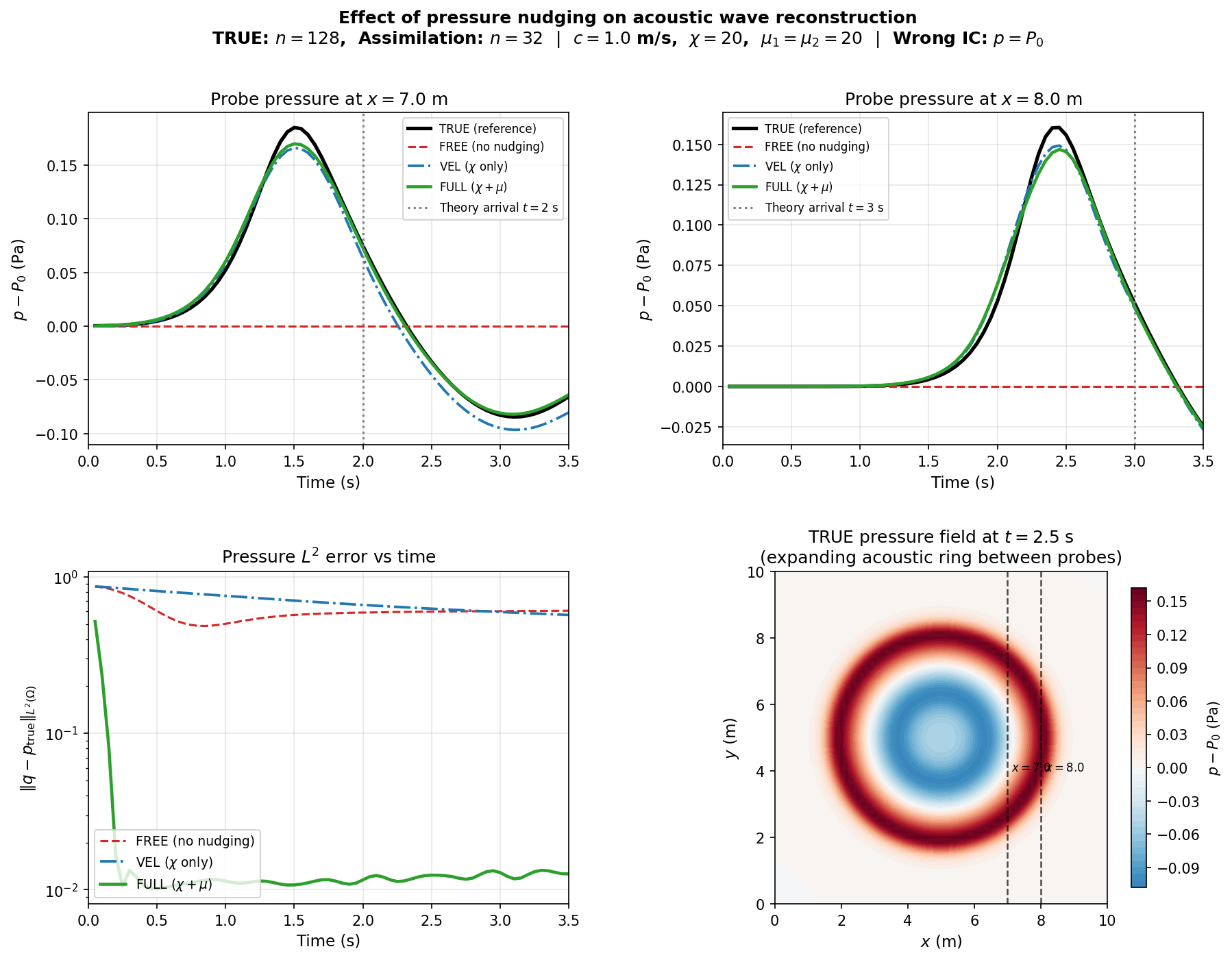}
    \caption{(Acoustic wave propagation) Acoustic wave reconstruction from wrong initial conditions ($p = P_0$, no pulse).
    \textbf{Top row}: probe pressure perturbation $q(x_p, y_c, t) - P_0$ at $x=7$~m
    (left) and $x=8$~m (right). The grey dotted vertical line marks the theoretical
    wave arrival time. FREE shows no reconstruction; VEL reconstructs the wave shape but with spurious bulk pressure errors; FULL closely tracks TRUE at both probes.
    \textbf{Bottom left}: $L^2$ pressure error over time on a logarithmic scale. FULL achieves a $97.9\%$ reduction relative to FREE at $T=3.5$~s.
    \textbf{Bottom right}: TRUE pressure perturbation field at $t=2.5$~s, showing the outward-propagating acoustic ring between the two probes.
    Parameters: $c=1$~m/s, $\delta P=1$~Pa, $\sigma=0.5$~m,
    $\chi=\mu_1=\mu_2=20$, $n_{\mathrm{true}}=128$, $n_{\mathrm{assim}}=32$,
    $\Delta t=0.05$~s.}
    \label{fig:acoustic}
\end{figure}


\section{Conclusions}
In this paper, motivated by the limitations of velocity-only nudging methods for slightly compressible flow, we design an algorithm that incorporates both velocity and pressure data from the slightly compressible flow and nudges both quantities into the assimilated incompressible NSE. Our analysis in Section \ref{continuous-error-estimates} shows that the model error (resulting from using compressible flow data to nudge an incompressible model) decays exponentially in the initial error, with an asymptotic residual of order $O(H)$. The analysis also suggests that the nudging parameter for pressure $\mu_1$ should be $O(1/H^2)$ to achieve effective assimilation. 

The analysis reveals that terms involving $\frac{4|\mu_2-\mu_1 |}{\mu_1}$ arise in the estimates, indicating sensitivity to the mismatch between the nudging parameters. In particular, the choice $\mu_1=\mu_2$ eliminates these terms and leads to a simplified setting, which we adopt in this work. A systematic study of the effects of differing values of $\mu_1$ and $\mu_2$ remains an open problem.

To further validate the effectiveness of our algorithm, we conduct a suite of numerical experiments, including a manufactured solution to test convergence rates (which confirm the expected optimal convergence), a modified Taylor–Green vortex experiment, and an acoustic wave propagation problem. The results confirm synchronization of the nudged system with slightly compressible observational data. In the modified Taylor–Green vortex experiment, additional validation is provided by flow statistics, including kinetic energy, vorticity, enstrophy, and divergence. The acoustic wave propagation problem demonstrates the necessity of pressure nudging for achieving synchronization of the pressure field.

This work provides a theoretical and numerical foundation for velocity–pressure nudging in slightly compressible flows, opening the door to further investigations of optimal pressure nudging parameter choices, error estimates for semi- and fully discrete cases, and applications in realistic compressible flows. 

\section*{Acknowledgments}
    The authors thank Professor William Layton for suggesting this problem and helpful discussions regarding this work.

\end{document}